\documentclass{article}

\usepackage[utf8]{inputenc}
\usepackage{amsmath,amssymb,amsthm}
\usepackage{tikz-cd}
\usepackage{mathtools}
\usepackage{enumitem}
\usepackage{bbold}
\usepackage{hyperref}

\newtheorem{definition}{Definition}[section]
\newtheorem{proposition}{Proposition}[section]

\newtheorem{example}{Example}[section]

\newtheorem{question}{Question}[section]
\newtheorem*{notation}{Notation}
\newtheorem*{acknowledgements}{Acknowledgements}
\newtheorem{remark}{Remark}[section]

\newcommand{\GZip}{\mathop{\text{$G$-{\tt Zip}}}\nolimits}
\newcommand{\GZipmu}{\mathop{\GZip^{\mu}}\nolimits}
\newcommand{\GF}{\mathop{\text{$G$-{\tt ZipFlag}}}\nolimits}
\newcommand{\GFmu}{\mathop{\GF^{\mu}}\nolimits}
\newcommand{\Hdg}{\mathop{\left[P\backslash *\right]}\nolimits}
\newcommand{\Hdgd}{\mathop{\left[P_{d}\backslash *\right]}\nolimits}

\author{Simon Cooper}
\title{Tautological rings of Hilbert modular varieties}
\begin{document}
	\maketitle
	
	\begin{abstract}
		We compute the tautological ring for a Hilbert modular variety at an unramified prime. The method generalises that of van der Geer from the Siegel case.
	\end{abstract}
	
	\section{Introduction}
	
	In this note we shall compute the tautological ring of Hilbert modular varieties in characteristic $p$.
	
	\vspace{3mm}
	
	For a Shimura variety $S_{K}$, its tautological ring $T^{\bullet}(S_{K})$ is the subring of its Chow ring with rational coefficients generated by the Chern classes of automorphic vector bundles. This ring was first defined and studied by van der Geer \cite{vanderGeer} in the Siegel case and later by Wedhorn-Ziegler \cite{Wedhorn-Ziegler} for Shimura varieties admitting an embedding into a Siegel Shimura variety (i.e. Hodge-type). Given an integral canonical model of a Shimura variety, conjectured to exist at hyperspecial level by Milne \cite{Milne} (and proven to exist for Hodge- and abelian- type by Kisin \cite{Kisin} and Vasiu \cite{Vasiu}), one can also define the tautological ring for the reduction of a Shimura variety to characteristic $p$. When suitable toroidal compactifications $S_{K}^{tor}$ exist there is the folk-lore conjecture that $$T^{\bullet}(S_{K}^{tor}) \cong H^{2\bullet}(\mathbf{X}^{\vee}(\mathbb{C}),\mathbb{Q})$$ where $\mathbf{X}^{\vee}$ is the compact dual. Van der Geer proved this in the Siegel case in all characteristics and Wedhorn-Ziegler proved that it holds for Shimura varieties of Hodge-type in characteristic $p$ (where suitable toroidal compactifications of integral models exist by work of Madapusi-Pera \cite{MadapusiPera}). We are interested in the related question:
	
	\begin{question}
		Is there a similar description of $T^{\bullet}(S_{K})$?
	\end{question} 

	In the Siegel case the answer is yes. Indeed, van der Geer proved that $T^{\bullet}(A_{g}) \cong H^{2\bullet}({\mathbf{X}_{g-1}}^{\vee})$, where $\mathbf{X}_{g-1}$ is the double Siegel half-space. In this note we compute $T^{\bullet}(S_{K})$ for a Hilbert modular variety in characteristic $p$ (see the beginning of section 5 for a precise statement).
	
	\newtheorem*{thm1}{Theorem 1}
	\begin{thm1}
		Let $X_{K}$ be a Hilbert modular variety at an unramified prime. $$T^{\bullet}(X_{K}) \cong H^{2\bullet}(\textbf{X}^{\vee})/(\text{top degree part})$$
	\end{thm1}  

	In particular, the only extra relation is the vanishing of the top Chern class of the Hodge vector bundle pulled back from the Siegel case. It is not clear what a general description for Hodge-type Shimura varieties could be. For example, $S_{K}$ could itself be compact. 
	
	\vspace{3mm}
	
	The proof of Theorem 1 generalises the method of van der Geer from the Siegel case, using properness of subvarieties to obtain positivity of certain classes. Pullback gives a surjective map of graded rings 
	$$H^{2\bullet}(X^{\vee}) \cong T^{\bullet}(S_{K}^{tor}) \longrightarrow T^{\bullet}(S_{K})$$
	Our aim is to determine the kernel of this morphism of graded rings. Section 2 recalls the Hilbert modular datum and in section 3 we review Wedhorn-Ziegler's definition of $T^{\bullet}(S_{K})$ and computation of $T^{\bullet}(S_{K}^{tor})$ for the Hilbert modular datum. In section 4 we choose an explicit Siegel embedding and explain that the pullback of the Hodge vector bundle splits into line bundles whose Chern classes generate $T^{\bullet}(S_{K}^{tor})$. This automatically gives a new relation coming from the vanishing of the top Chern class of the Hodge vector bundle in the Siegel case \cite[1.2]{vanderGeer}. The final section tackles the problem of showing that there are no further relations. This reduces to an analysis of the highest degree part, which is at most $d$-dimensional. There are $d$ elements given by multiples of $d-1$ generators and we show that they are linearly independent. This requires the properness of certain strata closures with a description of their cycle classes and is the novel difficulty in the Hilbert modular case: in the Siegel case the highest degree part is $1$-dimensional so it sufficed to find a non-zero class.
	
	\begin{acknowledgements}
		The author would like to thank W. Goldring for his support and guidance and D. Petersen for helpful discussions. The author is supported by grant KAW 2018.0356.
	\end{acknowledgements}
	
	\section{Hilbert modular datum}
	
	We begin by recalling the Shimura datum which gives a Hilbert modular variety and describing the cocharacter datum one obtains from this in characteristic $p$. 
	
	\vspace{3mm}
	
	Let $F$ be a totally real field of degree $d>1$ over $\mathbb{Q}$. Define $\mathbf{G}$ as the algebraic group  
	\[\begin{tikzcd}
	\mathbf{G} \arrow[r] \arrow[d]& \mathbb{G}_{m,\mathbb{Q}} \arrow[d]\\ Res_{F/\mathbb{Q}}GL_{2} \arrow[r, "det"]& Res_{F/\mathbb{Q}}\mathbb{G}_{m}
	\end{tikzcd}\] 
	Define $\mathbf{X}$ as the $\mathbf{G}(\mathbb{R})$-conjugacy class of the cocharacter $h \colon \mathbb{S} \rightarrow G_{\mathbb{R}}$ given on $\mathbb{R}$-points $\mathbb{C}^{\times} \rightarrow GL_{2}(\mathbb{R})^{d}$ by $$x+iy \mapsto \left(\begin{pmatrix}
	x & y \\ y & -x
	\end{pmatrix}\right)_{i \leq d}$$
	
	$(\mathbf{G},\mathbf{X})$ is a Shimura datum, henceforth referred to as the Hilbert modular datum given by $F$. The associated Shimura variety is the Hilbert modular variety and is of PEL-type. Indeed, Raynaud constructed a model over $\mathbb{Z}$ as a moduli space of principally-polarised abelian varieties with real multiplication by $\mathcal{O}_{F}$.
	
	\begin{definition}
		 Let $p$ be a rational prime unramified in $F$. Fix an algebraic closure $k$ of $\mathbb{F}_{p}$ and a hyperspecial level $K = K^{p}K_{p}$. Denote $X_{K}$ the special fibre of the integral model of the Hilbert modular variety at level $K$ and $X_{K,k}$ its base change to $k$. Fix a smooth proper toroidal compactification $X_{K}^{tor}$, which exists by \cite{MadapusiPera}.
	\end{definition}
	
	Let $G$ be the special fibre of a smooth reductive $\mathbb{Z}_{p}$-model of $\mathbf{G}$. We obtain a cocharacter datum $(G,\mu)$ as in \cite[1.2.2]{GoldringKoskivirta1}.  If $p$ is inert then $G$ is the fibre product
	\[\begin{tikzcd}
	G \arrow[r] \arrow[d]& \mathbb{G}_{m,\mathbb{F}_{p}} \arrow[d]\\ Res_{\mathbb{F}_{p^d}/\mathbb{F}_{p}}GL_{2} \arrow[r, "det"]& Res_{\mathbb{F}_{p^d}/\mathbb{F}_{p}}\mathbb{G}_{m}
	\end{tikzcd}\]
	and if $p$ splits then $G$ is the fibre product
	\[\begin{tikzcd}
	G \arrow[r] \arrow[d]& \mathbb{G}_{m,\mathbb{F}_{p}} \arrow[d]\\ (GL_{2,\mathbb{F}_{p}})^{d} \arrow[r, "det"]& (\mathbb{G}_{m,\mathbb{F}_{p}})^{d}
	\end{tikzcd}\]
	
	Fix an embedding $\mathbb{F}_{p}\hookrightarrow k$ so we can identify $\{\tau \colon \mathbb{F}_{p} \hookrightarrow k\} \cong \mathbb{Z}/d\mathbb{Z}$. For any unramified prime, the group $G_{\mathbb{F}_{p^{d}}}$ can be described as the $d$-tuples in $GL_{2}^{d}$ which have the same determinant in each factor. Any field extension $\mathbb{F}_{p^{d}} \subset K$ contains the field of definition $\kappa = \mathbb{F}_{p}$ of $\mu$ and
	$$\mu_{K} \colon \mathbb{G}_{m,K} \longrightarrow G_{K}$$ is given by $t \mapsto (\left(\begin{pmatrix}
	t & 0 \\0 & 1
	\end{pmatrix}\right)_{i\leq d},t)$.
	
	\vspace{3mm}
	
	Let $P$ be the parabolic subgroup associated with $\mu$ given on $k$-valued points $P(k) = \{(\begin{pmatrix}
	a_{i} & b_{i} \\ 0 & d_{i}
	\end{pmatrix}, a_{i}d_{i})\} \subset G(k)$ (this is called $P_{-}(\mu)$ in \cite{Wedhorn-Ziegler}, $\mathbf{P}^{-}$ in \cite{GoldringKoskivirta1}). This has a Levi subgroup $L = Cent(\mu) = \{(\left(\begin{pmatrix}
		a_{i} & 0 \\ 0 & d_{i}
		\end{pmatrix}\right), a_{i}d_{i})\}$ which is a maximal torus $T$.
	
	\section{Tautological ring of a toroidal compactification}
	
	In this section we review the computation of $T^{\bullet}(X_{K}^{tor})$.
	
	\vspace{3mm}
	
	For a Hodge-type Shimura variety $S_{K}$ in characteristic $p$, Wedhorn-Ziegler \cite{Wedhorn-Ziegler} proved that $T^{\bullet}(S_{K}^{tor}) \cong A^{\bullet}(\GZipmu)$, where $\GZipmu$ is the stack introduced in \cite{PinkWedhornZiegler} defined from the cocharacter datum $(G,\mu)$ and classifying certain semi-linear algebraic data. In characteristic $p$ the morphism $\sigma \colon S_{K} \rightarrow \Hdg$ given by the $P$-torsor on $S_{K}$ appearing in the $G$-zip $H^{1}_{dR}$ factors as:
	
	\[\begin{tikzcd}
	S_{K}^{tor} \arrow[drr, "\zeta^{tor}"]\\ & & \GZipmu \arrow[r, "\beta"] & \Hdg\\ S_{K} \arrow[uu, "j"] \arrow[urr, "\zeta"]
	\end{tikzcd}\]
	
	\begin{definition}
		Let $S_{K}$ be a Hodge-type Shimura variety in characteristic $p$. Its tautological ring is defined
		$T^{\bullet}(S_{K}) \coloneqq Im(\sigma^{*}) = Im(\zeta^{*}\circ \beta^{*}) \subset A^{\bullet}(S_{K})$. Similarly, for a toroidal compactification $T^{\bullet}(S_{K}^{tor}) \coloneqq Im(\zeta^{tor,*}\circ \beta^{*})$.
	\end{definition}In general, $\zeta$ is smooth (\cite[3.1.2]{Zhang}) while smoothness of $\zeta^{tor}$ is expected to follow from \cite{LanStroh} (and is known in the PEL case). Wedhorn-Ziegler proved that $\beta^{*}$ is surjective and $\zeta^{tor,*}$ is injective, giving $T^{\bullet}(S_{K}^{tor})\coloneqq Im(\zeta^{tor,*}\circ \beta^{*}) \cong A^{\bullet}(\GZipmu)$. 
	
	\vspace{3mm}
	
	Now let $(G,\mu)$ be the cochacter datum associated with $X_{K}$ introduced in section 2. Then
	$\pi \colon \GFmu \longrightarrow \GZipmu$ is an isomorphism, where $\GFmu$ is the stack of zip flags defined in \cite[2.1]{GoldringKoskivirta1}. Thus, on Chow rings $$A^{\bullet}(\GZipmu_{k}) \cong A^{\bullet}(\GFmu_{k})$$ The Weyl group $W = W_{G,T} = \{\pm 1\}^{\mathbb{Z}/d\mathbb{Z}}$ acts on the abelian group $X^{*}(T_{k})$ via its action on $T_{k}$. This extends to an action on the symmetric algebra $S \coloneqq Sym(X^{*}(T_{k})\otimes_{\mathbb{Z}}{\mathbb{Q}})$ and we denote $I = S_{+}^{W}$. Then \cite[8.2]{Wedhorn-Ziegler} gives $$A^{\bullet}(\GFmu_{k}) = S/IS\cong \mathbb{Q}[z_{1},\ldots,z_{d}]/(z_{1}^{2},\ldots,z_{d}^{2})$$ 
	 
	These generators $z_{i} \in S$ are given by characters $$z_{i}\colon\begin{pmatrix}
	a_{j} & 0 \\0 & d_{j}
	\end{pmatrix} \mapsto d_{i}$$ in $X^{*}(T_{k})$. Thus, $$T^{\bullet}(X_{K,k}^{tor}) \cong \frac{\mathbb{Q}[z_{1},\ldots,z_{d}]}{(z_{1}^{2},\ldots,z_{d}^{2})}$$
	
	\begin{remark}
		This description is essentially independent of the prime. However, one has an action of Frobenius $\sigma$ on $S$ and this depends on the prime. For example, if $p$ is inert then $\sigma \cdot z_{i} = z_{i+1}$ for all $i$, whereas if $p$ is split then the action is trivial. Use $\sigma$ to also refer to the corresponding permutation of $\{1,\ldots,d\}$. 
	\end{remark}

	\begin{remark}
		It is possible to interpret the relation $z_{i}^{2}$ geometrically. In the notation of section 5, there is a now-where vanishing section in $H^{0}(\overline{Z_{w_{i}}},L_{i}\otimes L_{\sigma(i)}^{p})$ giving $ 0 = c_{1}(L_{i}\otimes L_{\sigma(i)}^{p}|_{\overline{Z_{w_{i}}}}) = (z_{i}+pz_{\sigma(i)})(z_{i}-pz_{\sigma(i)})$ for each $i$. 
	\end{remark}
	
	\section{A Siegel embedding}
	
	\begin{notation} Denote $x_{i} \coloneqq j^{*}(z_{i}) \in T^{\bullet}(X_{K,k})$. Denote $V(\rho)$ the vector bundle on $[G\backslash *]$ given by a representation $\rho$ of $G$. 
	\end{notation}
	
	In this section we prove that the relation $x_{1}\ldots x_{d} = 0$ holds in $T^{\bullet}(X_{K,k})$. Fixing a Siegel embedding gives a description of $z_{i}$ in terms of the Hodge vector bundle. Namely, the pullback of the Hodge vector bundle $\mathbb{E}$ splits on $X_{K,k}$ into line bundles $\omega_{i}$ with $x_{i} = c_{1}(\omega_{i})$. The relation follows because $c_{d}(\mathbb{E}) = 0$ \cite[1.2]{vanderGeer}. Note that we have made a judicious choice of embedding to give precisely $x_{i} = c_{1}(\omega_{i})$; this is simply to ease computation. We will work with the following explicit description of the general symplectic group involved in the Siegel Shimura datum.
	
	\begin{example}
		Consider the symplectic vector space $V = \mathbb{F}_{p}^{2d}$ with symplectic form given by $J \coloneqq \begin{pmatrix}
		0 & -I_{d} \\ I_{d} & 0
		\end{pmatrix}$. 
		Then $GSp(V) = \{M = \begin{pmatrix}
		A & B \\C & D
		\end{pmatrix} \mid M^{T}JM = \lambda J \text{ for some } \lambda \in k\}$. Let $\mathcal{A}_{d,K'}$ the moduli of principally-polarised abelian varieties with level $K'$ hyperspecial at $p$. Then as in section 3 there is a morphism $\sigma_{d} \colon \mathcal{A}_{d,K'} \rightarrow \mathop{\text{$GSp(V)$-{\tt Zip}}}\nolimits \rightarrow \Hdgd$.
		
		\vspace{3mm}
		
		The Siegel parabolic $P_{d} = \{\begin{pmatrix}
			A & B \\0 & D
		\end{pmatrix}\}$ has a Levi quotient isomorphic to $GL_{d}\times \mathbb{G}_{m}$. The Hodge vector bundle on $\mathcal{A}_{d,K'}$ is $\mathbb{E}\coloneqq \sigma_{d}^{*}V(\rho_{d})$ given by the representation $\rho_{d} \colon P_{d} \twoheadrightarrow GL_{d}\times \mathbb{G}_{m} \xrightarrow{p_{1}} GL_{d}$.	Explicitly 
		$$\begin{pmatrix} A & B \\ 0 & D \end{pmatrix} \mapsto D$$
	\end{example}
		
	Now consider the standard embedding $\iota \colon G \hookrightarrow GSp(V)=GSp_{2d}$. This can be described by "embedded squares" over $k$:
	$$(\begin{pmatrix}
	a_{i} & b_{i} \\ c_{i} & d_{i}
	\end{pmatrix},\lambda) \mapsto \begin{pmatrix}
	A & B \\ C & D
	\end{pmatrix}$$ where $A = diag(a_{1}, \ldots , a_{d})$, $B = diag(b_{1}, \ldots , b_{d})$ etc.
	The parabolic $P \subset G$ maps into the Siegel parabolic $P_{d} \subset GSp(V)$.
		
	\vspace{3mm}
		
	We can explicitly study the pullback of the Hodge vector bundle to $X_{K,k}$. The following diagram commutes, giving $\iota^{*}\mathbb{E} = \sigma^{*}V(\rho)$.
	
	\[\begin{tikzcd}
	\mathcal{A}_{d,K'} \arrow[r]& \Hdgd \\ X_{K,k} \arrow[u, hook] \arrow[r]& \Hdg \arrow[u]
	\end{tikzcd}\]
	
	Explicitly, $\rho\colon P \rightarrow P_{d} \rightarrow GL_{d}$ is given by 
	
	$$\begin{pmatrix}
	a_{i} & b_{i} \\0 & d_{i}
	\end{pmatrix} \mapsto D$$ 

	where $D = diag(d_{1},\ldots d_{d})$. Consider the characters of $P$ given by $$\rho_{j}\colon\begin{pmatrix}
	a_{i} & b_{i} \\0 & d_{i}
	\end{pmatrix} \mapsto d_{j}$$
	
	Then $\rho = \rho_{1}\oplus \ldots \oplus \rho_{d}$ and so $V(\rho) \simeq V(\rho_{1}) \oplus \ldots \oplus V(\rho_{d})$ as vector bundles on $\Hdg$.
	Considering the commutative diagram 
	
	\[\begin{tikzcd}
	X^{*}(P) \arrow[d] \arrow[r]& X^{*}(T) \arrow[d]\\ A^{\bullet}(\Hdg)\arrow[r] & S \arrow[r]& S/IS \arrow[r]& T^{\bullet}(X_{K,k})
	\end{tikzcd}\] 
	we see that $x_{i} = \sigma^{*}c_{1}(V(\rho_{i}))$.
	
	\begin{proposition}
		$x_{1}\cdot \ldots \cdot x_{d} = 0$ in $T^{\bullet}(X_{K,k})$.
	\end{proposition}

	\begin{proof} 
		On $X_{K,k}$ we have $\iota^{*}\mathbb{E} = \sigma^{*}V(\rho)$ so
		$$x_{1} \ldots x_{d} = c_{d}(\sigma^{*}V(\rho)) = c_{d}(\iota^{*}\mathbb{E}) = \iota^{*}c_{d}(\mathbb{E}) = 0$$
		with the last equality given in \cite[1.2]{vanderGeer}.
	\end{proof}
	
	\section{Positivity of classes: proof of Theorem 1}
	
	In this section we prove that there are no further relations.
	
	\begin{thm1}
		Let $F$ be a totally real field, $p$ an unramified prime, $k$ an algebraic closure of $\mathbb{F}_{p}$, $K$ a hyperspecial level and $X_{K,k}$ the Hilbert modular variety at level $K$ in characteristic $p$.
		$$T^{\bullet}(X_{K,k}) = \frac{\mathbb{Q}[x_{1},\ldots,x_{d}]}{(x_{1}^{2},\ldots ,x_{d}^{2},x_{1}\cdot \ldots \cdot x_{d})}$$
	\end{thm1}
	
	The proof of Theorem 1 follows from:
	
	\begin{proposition}
		Let $p$ be a prime unramified in $F$. The $d$ classes $$\gamma_{i} \coloneqq x_{1}\ldots \widehat{x_{i}} \ldots x_{d}$$ are linearly independent in $A^{d-1}(X_{K,k})$.
	\end{proposition} 
	 
	The closures of Ekedahl-Oort strata $\overline{X_{K,w}} = \zeta^{-1}(\overline{Z_{w}})$ give cycles in $T^{\bullet}(X_{K})$ \cite[6.1]{Wedhorn-Ziegler}. The codimension $1$ strata are those indexed by $w_{i} = (\epsilon_{j}) \in \{\pm 1\}^{d}$ with $\epsilon_{j} = (-1)^{\delta_{ij}+1}$ (length $d-1$). There are sections $s_{i}$ of $L_{i}\otimes L_{\sigma(i)}^{-p}$, where $L_{i}\coloneqq \beta^{*}V(\rho_{i})$ on $\GZipmu$, such that $Z(s_{i}) = \overline{Z_{w_{i}}}$. To be precise, there is a separating system of partial Hasse invariants for the ordinary locus $Z_{ord}$ as defined in \cite[3.4.2]{GoldringKoskivirta2}. The divisors $D_{i} \coloneqq Z(\zeta^{*}s_{i}) = \overline{X_{K,w_{i}}}$ are the Goren-Oort strata closures defined and shown to be smooth and proper in \cite{GorenOort} and \cite{Goren}. Indeed, one can see by the modular description of $X_{K}^{tor}$ that $\overline{X_{K,w_{i}}}$ doesn't intersect the toroidal boundary: a semi-abelian variety with real multiplication is either abelian or ordinary. 
	
	\begin{proof}[Proof of Proposition 5.1]
		Denote $\gamma_{i}^{tor} \coloneqq z_{1}\ldots \widehat{z_{i}} \ldots z_{d}$ in $A^{d-1}(X_{K,k}^{tor})$ so that $\gamma_{i} = j^{*}\gamma_{i}^{tor}$. We obtain the following commutative diagram. 
		
		\[\begin{tikzcd}
		& X_{K,k}^{tor} \arrow[dr, "\zeta^{tor}"]\\ D_{i} \arrow[ur, hook] \arrow[dr, hook, "f_{i}"]& & \GZipmu_{k} \\ & X_{K,k}  \arrow[uu, hook', "j"] \arrow[ur, "\zeta"]
		\end{tikzcd}\]
		
		Assume that $\sum_{i=1}^{d}a_{i}\gamma_{i} = 0$ in $A^{d-1}(X_{K,k})$. 
		For each $i$ the projection formula \cite[2.6(b)]{Fulton} for the morphism $f_{i} \circ j$ applied to the class $\sum_{i=1}^{d}a_{i}\gamma_{i}^{tor}$ in $A^{d-1}(X_{K,k}^{tor})$ gives
		$$(f_{i}\circ j)_{*}(f_{i}\circ j)^{*}\sum_{i=1}^{d}a_{i}\gamma_{i}^{tor} = [D_{i}].\sum_{i=1}^{d}a_{i}\gamma_{i}^{tor}$$
		$(f_{i}\circ j)^{*}\sum_{i=1}^{d}a_{i}\gamma_{i}^{tor} = f_{i}^{!}\sum_{i=1}^{d}a_{i}\gamma_{i} = 0$ by assumption. Since $[D_{i}] = \zeta^{tor,*}[\overline{Z_{i}}] = z_{i}-pz_{\sigma(i)}$ this implies that for all $i \in \mathbb{Z}/d\mathbb{Z}$
		\begin{align*}
		0 &= [D_{i}].\sum_{i=1}^{d}a_{i}\gamma_{i}^{tor} \\&= (z_{i}-pz_{\sigma(i)}).\sum_{i=1}^{d}a_{i}\gamma_{i}^{tor} \\&= (a_{i}-pa_{\sigma(i)})(z_{1}\ldots z_{d})
		\end{align*}
		Since $deg(z_{1}\ldots z_{d}) = d > 0$ we have that $a_{i}-pa_{\sigma(i)}=0$ for all $i$. It remains to show that the only solution is $a_{i}=0$ for all $i$. For $p$ inert, $\sigma(i)=i+1$ (modulo $d$) and this system of linear equations is given by the circulant matrix
		$$A \coloneqq \begin{pmatrix}
			1 & -p & 0& \ldots& 0\\ 0 & 1 & -p & &   \\ \vdots& & \ddots& \ddots& \\ & & & &  \\ & & & 1 & -p \\ -p & & & & 1
		\end{pmatrix}$$
		A finite sequence of row operations gives an upper triangular matrix with diagonal entries $1,\ldots,1,1\pm p^{n}$ for some $n\neq 0$. The determinant of this matrix is non-zero, so the only solution to the system of linear equations is $$a_{1} = \ldots = a_{d} = 0$$ A similar argument yields the other cases.
		Hence, $\gamma_{1},\ldots ,\gamma_{d}$ are linearly independent.
	\end{proof}
	
	The proof of Theorem 1 can now be completed by showing that there are no relations in lower degrees (below $d$).

	\begin{proof}[Proof of Theorem 1]
		
		$T^{\bullet}(X_{K,k})$ is a quotient of $\frac{\mathbb{Q}[x_{1},\ldots,x_{d}]}{(x_{1}^{2},\ldots x_{d}^{2},x_{1}\cdot \ldots \cdot x_{d})}$. The degree of the generators $x_{i}$ is $1$ so the maximum degree is degree $d-1$. Suppose that there is a relation $$\sum\limits_{I \subset \{1,\ldots,d\},|I|=l}a_{I}z_{I}=0$$ in $T^{l}(X_{K,k})$, indexed over multisets, with $l < d$. For any $I_{0}$ there exists some $i \notin I_{0}$. Let $J = \{1,\ldots,i-1,i+1,\ldots,d\} \setminus I_{0}$. Multiplying by $z_{J}$ gives $\sum_{I}a_{I}z_{I \sqcup J}=0$ in $T^{d-1}(X_{K,k})$. This can be rewritten uniquely as $\sum_{I}a_{I}z_{I \sqcup J} = \sum b_{j}\gamma_{j}$ by Proposition 5.1. If $i$ is not in  $I \sqcup J$ then either $I=I_{0}$ or there is a repeated index so $z_{I \sqcup J} = 0$. Hence, $a_{I_{0}}$ is the only potentially non-zero coefficient of $\gamma_{i}$ and by linear independence in $T^{d-1}(X_{K,k})$ we have $a_{I_{0}} = 0$. Since this holds for each $I_{0}$ we have that the relation is trivial. 
		
		\vspace{3mm}
		
		Thus, there are no relations in lower degree and $T^{\bullet}(X_{K,k}) \cong \frac{\mathbb{Q}[x_{1},\ldots,x_{d}]}{(x_{1}^{2},\ldots x_{d}^{2},x_{1}\cdot \ldots \cdot x_{d})}$
		
	\end{proof}

	\bibliographystyle{alpha}
	\bibliography{refs}
	
	\end{document}